\documentclass[11pt]{amsart}
\usepackage{geometry}                % See geometry.pdf to learn the layout options. There are lots.
\usepackage{graphicx}
\usepackage{epstopdf}
\usepackage{amsthm}
\usepackage{amsmath}
\usepackage{amssymb}
\usepackage{footmisc}
\usepackage{setspace}

\newtheorem{theorem}{Theorem}[section]
\newtheorem{lemma}[theorem]{Lemma}

\theoremstyle{definition}
\newtheorem{definition}[theorem]{Definition}

\theoremstyle{remark}

\newtheorem{thm}{Theorem}[section]

\newtheorem{corollary}[thm]{Corollary}

\theoremstyle{definition}

\theoremstyle{remark}

\def\offset{a}
\def\Mnminus{M^-_n}
\def\Mnplus{M^+_n}
\def\Cnminus{C^-_n}
\def\Cnplus{C^+_n}

\title[On the minimum of the mean-squared error in 2-means clustering]{On the minimum of the mean-squared error\\ in 2-means clustering}
\author{Bernhard G. Bodmann and Craig J. George}\thanks{This paper was supported in part by an REU portion of NSF grant DMS-1412524.}
\address{Mathematics Department, University of Houston, Houston, TX 77204-3008.}
\begin{document}

\begin{abstract}
We study the minimum mean-squared error for 2-means clustering
when the outcomes of the vector-valued random variable to be clustered are on two touching spheres 
of unit radius in $n$-dimensional Euclidean space
and the underlying probability distribution is the normalized surface measure. 
For simplicity, we only consider the asymptotics of large sample sizes and 
replace empirical samples by the probability measure.
The concrete question addressed here is whether a minimizer for the 
mean-squared error identifies the two individual spheres as clusters.
Indeed, in dimensions $n \ge 3$, the minimum of the mean-squared error is achieved by a partition that separates the two spheres and has unit distance between the points in each cluster and the respective mean.
In dimension $n=2$, however, the minimizer fails to identify the individual spheres; an optimal partition is obtained by a separating hyperplane that
does not contain the point at which the spheres touch.
\end{abstract}
\maketitle

\section{Introduction}

In many applications of data science, large sets of vectors need to be grouped into a small
number of subsets whose elements are close to each other. This type of partitioning into subsets is also called clustering \cite{MacKay}. The subsets are often believed to be 
distinct constituents in a mixture of random vectors that are sampled from different distributions. 
In many cases, the distributions are from a known family that is parametrized by
the expected value of the outcomes, and the outcomes concentrate near the expected value \cite{POL,Das99}. Partitioning the observed set
of vectors into subsets yields the empirical means, also called centroids, which provide an estimate for the expected values. On the other
hand, once the expected values are accurately determined, one assumes that mapping each vector
to the subset whose centroid is closest provides a good partition.
This heuristic approach to the clustering problem is captured in an iterative algorithm by Lloyd \cite{LLOYD}, which 
aims to minimize an objective function that measures the Euclidean mean squared distance of the 
elements in each of the subsets from the respective centroid. Although the algorithm seems to work well in practice,
known results lack general a priori performance guarantees \cite{BUWI,KIE,Selim84,DuFG99,LuZ16}
or show cases with slow convergence \cite{Vattani11} even for two-dimensional clustering. 

Another setting in which one tries to minimize the mean-squared distance is in vector quantization \cite{BER,GEGR}, see also \cite{Steinhaus}. There, 
partitioning of the outcomes of a random vector is not explicitly motivated by an underlying assumption that
it is a mixture. The main goal is to approximate the random vector by a quantized one, with a finite or discrete set
of outcomes while minimizing the distortion, measured in the expected Euclidean squared norm of the
quantization error or in terms of more general norms \cite{GRLU}.  

In this paper, we investigate the problem of minimizing the objective function appearing in Lloyd's algorithm
for the special case of partitioning into two subsets. Optimality for the 2-means problem has already been considered in dimension $n=2$
for the concrete examples of the uniform distribution on the disk and on the square \cite{Roychowdhury}.
We consider the concrete example of random vectors governed by a probability measure $\rho$ that is
formed by taking the average of two probability measures that are uniform on two touching spheres of unit radius
in $n$-dimensional Euclidean space. The concrete question is then whether an optimizer to the mean-squared error
of 2-means clustering assigns, up to sets of measure zero, a partition that singles out each individual sphere.
Earlier
results indicate that a convex relaxation of Lloyd's clustering algorithm \cite{PengW07} is successful if the spheres are sufficiently separated \cite{iguchi2015tightness, IguchiMPV15,LLLSW}. 
Indeed, in dimension $n=1$,  the desired result is achieved if and only if  the spheres are separated
by a sufficiently large distance. A unit sphere in dimension $n=1$
is a set of two points at a distance of 2. The uniform probability measure on two symmetrically arranged spheres 
at a distance $2\epsilon$
is $\rho = (1/4) \delta_{-2-\epsilon} + (1/4)\delta_{-\epsilon} + (1/4)\delta_{\epsilon} + (1/4)\delta_{2+\epsilon}$, where $\delta_a$ is for any $a \in \mathbb R$ a Dirac measure with support $\{a\}$.
If we choose $0<\epsilon<(\sqrt 3 - 1)/2$, then by exhausting all choices of partitions,
it is seen that the set $S_1=\{-\epsilon, \epsilon, 2+\epsilon\}$ with mean $m_1=(2+\epsilon)/3$ and the set $S_2 = \{-2-\epsilon\}$ with mean $m_2=-2-\epsilon$
provide an optimal partition of $\{-2-\epsilon,-\epsilon,\epsilon,2+\epsilon\}$ for which the resulting mean-squared error
is $2(1+\epsilon+\epsilon^2)/3<1$, whereas the symmetric choice $R_1=\{\epsilon,2+\epsilon\}$ 
and $R_2=\{-\epsilon,-2-\epsilon\}$ gives a mean-squared error of $1$.
On the other hand, if $\epsilon>(\sqrt 3 - 1)/2$, then the partitioning into $R_1$ and $R_2$ is indeed optimal
for the mean-squared error.

It is tempting to attribute the failure to recover the individual spheres
to the discrete nature of the ``surface'' measures in $\mathbb R$. A closer look shows that the concentration
of the measure near the origin is the reason for the optimal partition formed by one sphere cannibalizing the other.  
As $n$ grows, the measure $\rho$ is less concentrated near the origin, and one expects
this cannibalizing behavior to disappear. 
Here, we examine the question whether a successful partition can be obtained 
in dimensions $n\ge 2$ {even if the spheres touch}. This is the most challenging case in which separation
can still be achieved theoretically. We consider the continuum limit, which means instead
of sampling the distributions with finitely many outcomes, we assume data given in the form of uniform measures on the spheres.

%The mixture of random vectors is obtained from uniform distributions
%on two spheres of a fixed radius in Euclidean $n$-dimensional space. In this setting, 

Our results show that
minimizing the mean-squared error in $\mathbb{R}^2$ leads to a non-symmetric partition, as in the case of dimension $n=1$. Fortunately, in dimensions $n \ge 3$
the minimizer recovers the partition into individual spheres, as one hopes to achieve.

This paper is organized as follows: In Section~\ref{sec:optpart}, we present the main results. The proofs are 
either elementary and included there or relegated to the Section~\ref{sec:proofs}. A first part of the proofs establishes
that optimal partitions for 2-means clustering are obtained from separating hyperplanes. The next part determines the location of the hyperplane.

{\bf Acknowledgment.} Both authors would like to thank Dustin Mixon for suggesting the 
intriguing calculus exercise worked out in Section~\ref{subsec:offset}.

%\section{}
%\subsection{}

%\title{Proof for $E$ Increasing in $\offset \in (0,2)$}
%\author{Craig Joseph George}
%\date{\today}

%documentclass[10pt]{article}

%\begin{document}
%\maketitle
%\doublespacing

%\newpage %%PAGE 2

\section{Optimal partitions for the mean-squared error}\label{sec:optpart}

The problem we are concerned with is the minimization of  the mean-squared error. 
Its value
depends on the partition of the support of a probability measure $\rho$ describing the outcomes of a mixture of 
random vectors. 

\begin{definition} \label{def:MSE}
Given a Borel probability measure $\rho$ on $\mathbb R^n$ with support $S$
and a Borel-measurable subset $S_1 \subset S$ with complement $S_2 = S \setminus S_1$,
then the {\it mean squared error} associated with the partition $\{S_1,S_2\}$ of $S$
is
$$
   \mathcal{E} (S_1) =  \min_{c_1\in \mathbb{R}^n} \int_{S_1} \|x-c_1\|^2 d \rho(x) +  \min_{c_2 \in \mathbb{R}^n}\int_{S_2} \|x-c_2\|^2 d \rho(x)
$$ 
\end{definition}

In this paper, we are concerned with a special case where
$\rho$ is the (normalized) surface measure for the union of two touching spheres,
$$
   \rho = \frac 1 2 (\sigma_{-1} + \sigma_1 )\, .
$$
Here $\sigma_a$ is the surface measure supported on $\mathbb{S}_a \equiv \{x \in \mathbb{R}^n: \|x-a e_1\|=1\}$, where $e_1$ is the first
canonical basis vector in $\mathbb{R}^n$. The measure $\sigma_a$ is obtained from translating $\sigma_0$,
so for any Borel measurable set $A$, $\sigma_a (A+ae_1)= \sigma_0(A)$,
 and for any orthogonal matrix $O$,
$\sigma_0(A)=\sigma_0(O^{-1}(A))$.

The following are the main theorems in this paper:

  \begin{theorem}\label{thm:optimalhyperplane}
  Let the Borel measure be given by $\rho=  \frac 1 2 (\sigma_{-1} + \sigma_1 )$ on $\mathbb R^n$ with support $S=\mathbb{S}_{-1} \cup \mathbb{S}_1$.
  Let $S_1, S_2$ form a partition of $S$ into two Borel measurable subsets, then
  there exist $a \in \mathbb R$ and  $T_1 = \{x \in \mathbb R^n: x_1 \le a\}$ such that
  $\mathcal{E}(T_1) \le \mathcal{E}(S_1)$. Moreover, if $S_1$ is minimal for the mean-squared error, then 
  there is a choice of the cutoff $a$ for which $T_1$ coincides with $S_1$ or $S_2$, up to a set 
   of zero probability.
  \end{theorem}
  
  In short, disregarding sets of zero probability, an optimal partition of $S$ is given by two
  sets separated by a hyperplane orthogonal to $e_1$, at an offset $a$ from the origin.
  The fact that an optimal partition comes from a separating hyperplane is well known \cite{DuFG99},
  which we supplement with a symmetrization argument.
  
  This result motivates abbreviating the mean-squared error for this special case,
  and studying its dependence on the cutoff,
  $$
E(n, \offset ) = \mathcal{E}(\{x \in S: x_1 \le -\offset \}) \, .
$$
By the reflection symmetry of $\rho$ with respect to the first coordinate, it is sufficient to consider  $E(n,a)$ for $a \ge 0$. With this simplification, we can study the case of dimension $n=2$ in elementary terms.

\begin{theorem} {\it In dimension $n=2$, the absolute minimum of $E(2,\offset)$ among $\offset \in [0,2)$ is attained at a non-zero cutoff $\offset$.}
\end{theorem}
\begin{proof}
Parametrizing the two circles by arc length gives by a direct computation for $a=1-\frac{\sqrt{3}}{2}$ the 
probabilities
$
   \rho(\{x \in \mathbb{R}^2: x_1 \le -1+\frac{\sqrt 3}{2} \}) = \frac{5}{12}
$ and
$ 
  \rho(\{x \in \mathbb{R}^2: x_1 > -1+\frac{\sqrt 3}{2} \} ) = \frac{7}{12} .
$
Choosing $c_1 = (\zeta_1,0)$ and $c_2 = (\zeta_2,0)$
with $\zeta_1 = -1-\frac{3}{5\pi}$ and $\zeta_2 = \frac{5}{7}+ \frac{3}{7\pi}$
gives for the mean-squared error
\begin{align*}
   E\bigl(2,1-\mbox{$\frac{\sqrt{3}}{2}$}\bigr) \le \, & \frac{1}{4\pi} \biggl( \int_{\pi/6}^{11\pi/6} ((-1+\cos t - \zeta_1)^2 + \sin^2 t ) dt \biggr. \\
   & \biggl. + \int_{-\pi/6}^{\pi/6} ((-1+\cos t - \zeta_2)^2 + \sin^2 t) dt
    + \int_0^{2\pi} ((\cos t + 1 - \zeta_2)^2 + \sin^2 t) dt \biggr) \\
= &\,  \frac{45\pi^2 -30\pi -9}{35\pi^2} <0.987 \, .
\end{align*}
This is less than $E(2,0)=1$, so the absolute minimum is not attained at $a=0$. 
 \end{proof}

When the means of the two subsets $\{x \in \mathbb{R}^2: x_1 \le -a \}$ and $ \{x \in \mathbb{R}^2: x_1 > -a \} $
then Theorem~\ref{thm:optimalhyperplane} reduces identifying the optimal mean-squared error
to finding the minimum of a parameter integral.

In dimension $n=3$, the mean-squared error can be computed explicitly.

%can be computed in explicit terms, then using that they are minimizing in the Definition~\ref{def:MSE}
%according to Theorem~\ref{thm:optimalhyperplane} gives an expression for the mean-squared error.

\begin{theorem}\label{thm:uniqueminimumdim3}
 {\it In dimension $n=3$, the absolute minimum of $E(3,\offset)$ among $\offset \in [0,2)$ occurs at $\offset=0$.}
 \end{theorem}
\begin{proof} We parameterize the two spheres by spherical coordinates and normalize the measure by surface area. 
Based on Theorem~\ref{thm:optimalhyperplane},
an optimal partition is obtained with a separating hyperplane orthogonal to the symmetry axis $\mathbb{R} e_1$.
The associated probabilities are for $-2 \le a \le 2$:  
$
   \rho(\{x \in \mathbb{R}^2: x_1 \le -a \}) = \frac 1 2 - \frac{a}{4}
$ and
$ 
  \rho(\{x \in \mathbb{R}^2: x_1 > -a \} ) = \frac 1 2 + \frac{a}{4} .
$
As shown in Theorem~\ref{thm:sephyperplane} below,
the mean-squared error is obtained by choosing $c_1$ and $c_2$ to be the means of the two subsets,
 $c_1=(\zeta_1 ,0,0)$, $c_2=(\zeta_2 ,0,0)$ with $\zeta_1 = -1-\frac{1}{2} a$, $\zeta_2 =1-\frac{1}{2} a$. 
 This choice results in
\begin{align*}
   E(3,a) = \, &\frac{1}{8\pi} \biggl(  \int_{0}^{2\pi} \int_{\arccos(1-a)}^{\pi} ((-1+\cos u - \zeta_1)^2 + \sin^2 u )\sin u \, du dt \biggr. \\
   &+ \int_{0}^{2 \pi} \int_{0}^{\arccos(1-a)} ((-1+\cos u - \zeta_2)^2 + \sin^2 u )\sin u \, du dt \biggr. \\
   &+ \int_{0}^{2 \pi} \int_{0}^{\pi} ((1+\cos u - \zeta_2)^2 + \sin^2 u )\sin u \, du dt \biggr) \\
= &\,  \frac{1}{4} a^2 +1\, .
\end{align*}
Thus $E(3,a)$ achieves its absolute minimum at $a=0$.
\end{proof}

Even in the absence of explicit computations for $E(n,a)$ in case $n>3$, we obtain the same monotonicity property as for $n=3$.

\begin{theorem} \label{thm:uniqueminimum}
{\it The inequality $\frac{\partial}{\partial \offset} E (n,\offset )> 0$ holds for all $\offset \in (0,2)$ and $n>3$. Moreover, $E(n,\offset )$ attains a minimum at $\offset =0$, and this minimum is unique.}
\end{theorem}

Theorems~\ref{thm:uniqueminimumdim3} and \ref{thm:uniqueminimum} give us that the $2$-means objective function $E$ of two touching $n$-spheres is increasing in the variable $\offset$ for the cutoff for $n\ge 3$ in the continuum limit. Thus, for dimensions $n\ge 3$, the optimal $2$-means cutoff has a value of zero, so both $n$-spheres are recovered successfully.

The remainder of the paper is dedicated to the outstanding proofs of Theorems~\ref{thm:optimalhyperplane} and~\ref{thm:uniqueminimum}. 

\section{Proofs of main results on optimal partitions} \label{sec:proofs}

The first part of this section establishes the proof that an optimal partition
is given by a separating hyperplane that is orthogonal to the symmetry axis.
The second part examines the offset of the optimal separating hyperplane.

\subsection{Minimizing the mean-squared error by partitions with a separating hyperplane}

First, we consider a general Borel measure $\rho$ with support $S$ in $\mathbb{R}^n$.
Given a partition $\{S_1,S_2\}$ of $S$,
and $\rho(S_i) > 0$, then we call $m(S_i) = \int_{S_i} x d\rho(x) / \rho(S_i)$ the {\it mean} associated with the set $S_i$.
If $S_i$ is clear from the context, we also abbreviate $m_i = m(S_i)$.

By a direct computation, we have for any $S_i$ with $\rho(S_i)>0$ and $c_i \in \mathbb{R}^n$
$$
  \int_{S_i }\|x-c_i\|^2 d\rho(x)  = \int_{S_i }\|x-m_i\|^2d\rho(x) + \|c_i - m_i\|^2 \rho(S_i) \, ,
$$
so the minimum is achieved if and only if $c_i = m_i$.

Moreover, given $c_1, c_2 \in \mathbb{R}^{n}$ , then among all the partitions, the partition into Voronoi regions
is optimal, as shown in Lemma~\ref{lem:voronoi} below.

\begin{definition}
Given $c_1, c_2 \in \mathbb R^{n}$, we define the {\it Voronoi partition} $\{T_1, T_2\}$ of a Borel set $S$ associated with the vectors $c_1$ and $c_2$  by the assignment
$$
   T_1 = \{x \in S: \|c_1 - x\| \le \|c_2 - x\| \} \, , T_2 = S \setminus T_1 \, .
$$   
\end{definition}

From this definition, we see that this Voronoi partition consists of a closed half-space and its complement,
with a separating hyperplane that is orthogonal to $c_1-c_2$ and contains the midpoint $(c_1+c_2)/2$.

Next, we note that given a partition into sets of non-zero probability, passing 
to the Voronoi partition associated with the means can only improve the mean-squared error.

\begin{lemma} \label{lem:voronoi}
Let $S_1,S_2$ be a partition of $S$ with $0<\rho(S_1)<1$ and associated means $m_1$ and $m_2$, then
the Voronoi partition associated with $m_1, m_2$ satisfies
$$  
   \mathcal{E} (T_1) \le \mathcal{E}(S_1) \, .
$$ 
\end{lemma}

\begin{proof}
For any measurable partition $S_1$ and $S_2$ and $i \in \{1, 2\}$, 
choosing any $x \in T_i$ gives
by the definition of the Voronoi partition
$\|x-m_i\| \le \min \{\|x-m_1\|,\|x-m_2\|\}$.  Thus, the partition
of $S$ into $T_1$ and $T_2$ gives a mean-squared error that is bounded
above by that associated with $S_1$ and $S_2$.
\end{proof}

In the following, we focus on properties of optimal partitions. These properties are already known even in a more general context of $k$-means \cite{DuFG99,GRLU}.
We have decided to include them here to keep the exposition self-contained.

\begin{lemma}
If $\{S_1, S_2\}$ is a minimizing partition for the mean-squared error, then 
$0<m(S_i)<1$ for $i \in \{1,2\}$ and
$m(S_1) \ne m(S_2)$.
\end{lemma}
\begin{proof}
Let $\{S_1, S_2\}$ be a minimizing partition. We know $0<\rho(S_1)<1$, otherwise $S_1$ or $S_2$ have 
unit measure and we can
refine $S_1$ or $S_2$ and improve the mean-squared error.
%If one of the sets has probability one, say $\rho(S_1)=1$, then then the second term in the mean-squared
%error vanishes. If $m_1=m(S_1)$, then choosing any subset $R_1\subset S_1$ and $R_2 = S \setminus R_1$
%gives
%$$
% \mathcal{E}(S_1) = \int_{R_1} \|x-m_1\|^2 d\rho(x) + \int_{R_2} \|x-m_1\|^2 d\rho(x) \, .
%$$ 
%By the assumption of minimality, and with a set $R_1$ of probability $0<\rho(R_1)<1$, we then know 
%that $m_1=m(R_1)=m(R_2)$. Thus, we obtain an optimal partition with sets of non-zero probabilities
%and equal means. 

Moreover, assuming an optimal partition into two sets $S_1$ and $S_2$ of non-zero probability and equal means 
$m_1=m_2$, then any partition 
performs equally well, and we can choose a subset $R_1\subset S_1$ with $0<\rho(R_1) <1$
such that the associated mean $r_1\equiv m(R_1) \ne m_1$. By the characterization of the mean,
then $\int_{R_1} \|x-r_1\|^2 d\rho(x) < \int_{R_1} \|x-m_1\|^2 d\rho(x)$.
For the partition formed by $R_1$ and $R_2 = S \setminus R_1$, we then get that 
\begin{align*}
  \int_{R_1} \|x-r_1\|^2 d\rho(x) + \int_{R_2}\|x-m_1\|^2 d\rho(x) & < \int_{R_1} \|x-m_1\|^2 d\rho(x) + \int_{R_2} \|x-m_1\|^2 d\rho(x) \\
  & = \mathcal{E}(S_1) \, .
\end{align*}
Now inserting the mean of $R_2$ instead of $m_1$ in the second term on the left shows that
$$
   \mathcal{E} (R_1) < \mathcal{E}(S_1) \, .
$$
This contradicts optimality, so $m_1=m_2$ cannot hold for a minimizing partition.
\end{proof}

\begin{theorem}\label{thm:sephyperplane}
Let $\rho$ be a Borel measure on $\mathbb{R}^n$ with support $S$.
If the partition $\{S_1, S_2\}$ is a minimizer for the mean-squared error, then the sets $T_1$ and $T_2$ in the 
Voronoi partition associated with the means $\{m(S_i)\}_{i=1}^2$ coincide with $S_1$ and $S_2$ up to
changes involving subsets of the  separating hyperplane or sets whose probability vanishes.
\end{theorem}
\begin{proof}
We know $0<\rho(S_1)<1$, so both sets $S_1$ and $S_2$ have means under $\rho$.

Passing to the Voronoi partition $\{T_1, T_2\}$ associated with
these means $\{m(S_i)\}_{i=1}^2$ gives
$$
  \mathcal{E} (T_1) = \mathcal{E} (S_1) \, .
$$
Using the inequality in the definition of the Voronoi partition,
we see that if $R_1 = T_1 \cap S_2$ is non-empty, then so is $R_2 = T_2 \cap S_1$, and 
 $ \|x-m_i\| \le \min \{\|x-m_1\|,\|x-m_2\|\}$ if $x \in R_i \subset T_i$, $i \in \{1, 2\}$.
Hence, denoting the hyperplane  $H= \{x \in \mathbb{R}^n: \|x-m_1\| = \|x-m_2\|\}$,
on $R_1\setminus H$ and $R_2 \setminus H$ strict inequality holds in the norm bounds,
and we see that by the monotonicity of integrals, the equality $\mathcal{E}(T_1) = \mathcal{E}(S_1)$ forces both sets 
to have probability zero, $\rho(R_1\setminus H)= \rho(R_2 \setminus H) = 0$.  
\end{proof}

  From now on, we specialize to $\rho=(\sigma_{-1}+\sigma_1)/2$.
  As a first result for this concrete choice of $\rho$, we show that the mean-squared error does not increase when passing to a suitable partition into half-spaces
  that are separated by a hyperplane orthogonal to $e_1$.
  
  To obtain this, we note that choosing a partition that separates into half-spaces with a separating hyperplane that
  contains the symmetry axis $\mathbb{R} e_1$ is not optimal. Without loss of generality, we orient this hyperplane
  so that it is orthogonal to $e_2$.
  
  \begin{lemma} Let $n \ge 2$, $\rho= \frac 1 2 (\sigma_{-1} + \sigma_1 )$ be the measure defined on $\mathbb{R}^n$
  with support $S$, $S_1 = S \cap \{x \in \mathbb{R}^n: x_2 \ge 0\}$ and
  $T_1 = S \cap \{x \in \mathbb{R}^n: x_1 \ge 0\}$, then
  $\mathcal{E}(S_1)>\mathcal{E}(T_1)$.
  \end{lemma}
  \begin{proof}
  By symmetry, the mean of $S_1$ is $m(S_1) = \alpha e_2$. Also, we know that the mean is in the interior of the convex hull of $S_1$,
  so  $0< \alpha < 1$. Again using the symmetry between $S_1$ and $S_2$ as well as $\rho(S_1) = \rho(S_2)=1/2$, 
  $$
    \mathcal{E}(S_1) = 2 \int_{S_1} \|x-\alpha e_2\|^2 d\rho = 2 \int_{S_1} \|x\|^2 d\rho - \alpha^2  = \int_S \|x\|^2 d\rho - \alpha^2 \, .
  $$
  Next, comparing with the Voronoi partition corresponding to $\{\pm e_1\}$ and using symmetry properties, we have
  $$
    \mathcal{E}(S_1) = 2 (\int_{T_1} \|x-e_1\|^2 + 1/2) - \alpha^2 = \mathcal{E}(T_1) +1 - \alpha^2 \, .
  $$
  From $0<\alpha<1$, we then have $\mathcal{E}(S_1) >\mathcal{E}(T_1)$.
  \end{proof}
  
  We are now ready to prove Theorem~\ref{thm:optimalhyperplane}, which states that an optimal partition coincides,
  up to sets of measure zero, with one obtained from a separating hyperplane that is orthogonal to $\mathbb R e_1$.
  
  %\begin{theorem}
 % Let the Borel measure be given by $\rho=  \frac 1 2 (\sigma_{-1} + \sigma_1 )$ on $\mathbb R^n$, $n \ge 2$, with support $S=\mathbb{S}_{-1} \cup \mathbb{S}_1$.
  %Let $S_1, S_2$ form a partition of $S$ into two Borel measurable subsets, then
  %there exist $a \in \mathbb R$ and  $T_1 = \{x \in \mathbb R^n: x_1 \le a\}$ such that
  %$\mathcal{E}(T_1) \le \mathcal{E}(S_1)$. Moreover, if $S_1$ is minimal for the mean-squared error, then $S_1$ or $S_2$ coincide,
  %up to a set of zero probability, with a set $T_1$ given by a cutoff $a \in \mathbb R$.
  %\end{theorem}
  \begin{proof}[Proof of Theorem~\ref{thm:optimalhyperplane}]
  Given a partition of $S$ by $S_1$ and $S_2$ with means $m_i = m(S_i)$, $i \in \{1 ,2 \}$,
  we observe the following:
  
  The algebra of Borel sets of the form $A_1 \times \mathbb{R}^{n-1}$ with $A_1 \subset \mathbb R$,
  is a sub-algebra of the Borel algebra of $\mathbb{R}^n$. The functions
  that are measurable with respect to this algebra depend only on the first coordinate.
  By the Radon-Nikodym theorem, there exist functions $d_i: \mathbb{R} \to \mathbb{R}$
  such that for any $A=A_1 \times \mathbb{R}^{n-1}$,
  $$
    \int_A d_i(x_1) \chi_{S_i}(x) d\rho(x) = \int_A \|x-m_i\|^2 \chi_{S_i}(x) d\rho(x) \, .
  $$
  Next, using Fubini, if $\mu$ is the image measure of $\rho$ under projection onto
  the first coordinate, $\mu(A_1) = \rho(A_1 \times \mathbb{R}^{n-1})$, then
  there is $f: \mathbb R \to [0,1]$ such that
  $$
     \int_{A_1} d_1f d\mu  = \int_{A_1 \times   \mathbb{R}^{n-1}} \|x - m_1\|^2 \chi_{S_1}(x) d\rho(x) \, 
  $$
  and
  $$ \int_{A_1} d_2 (1-f) d\mu  = \int_{A_1   \times \mathbb{R}^{n-1}} \|x - m_2\|^2 (1-\chi_{S_1}(x)) d\rho(x) \, .
  $$
  Next, we observe if $f$ is the function associated with a partition $S_1$ and $S_2$ and $R_1 = \{x \in \mathbb R: d_1(x) \le d_2(x)\}$,
  then letting $g = \chi_{R_1}$ gives that
  $$
      \int_{\mathbb R} d_1 g d\mu + \int_{\mathbb R} d_2 (1-g)  d\mu
      \le   \int_{\mathbb R} d_1 f d\mu + \int_{\mathbb R} d_2 (1-f)  d\mu \, .
  $$
  We conclude,
  setting $T_1' = S \cap (R_1 \times \mathbb{R}^{n-1})$ and $T_2' = S \setminus T_1'$ that
  $$
     \int \|x - m_1\|^2 \chi_{T_1'} d\rho + \int \|x - m_2\|^2 \chi_{T_2'} d\rho
      \le \mathcal{E}(S_1) \, .
      $$
     Next, replacing $m_1$ and $m_2$ by the means $m_i' \equiv m(T_i')$, $i \in \{1,2\}$,  does not increase the left-hand side, which shows that
     $$
     \mathcal{E}(T_1') \le \mathcal{E}(S_1)\, .
  $$
  Finally, setting $\{T_1,T_2\}$ to be the Voronoi partition associated with the means $m_1'$ and $m_2'$
 implies
  $$
    \mathcal{E}(T_1) \le \mathcal{E}(S_1) \, .
  $$
  
  Moreover, if $S_1$ is chosen as a minimizer for the mean-squared error, then necessarily $m_i = m_i'$, $i \in \{1,2\}$,
  otherwise we would have strict inequality between $\mathcal{E}(T_1')$ and $\mathcal{E}(S_1)$. This implies 
  that the means $m_i$ are on the symmetry axis $\mathbb R e_1$. 
  Applying Theorem~\ref{thm:sephyperplane}
  now shows that, up to a set of probability zero, $S_1$ and $S_2$ are separated by a hyperplane.
  From the preceding lemma, optimality implies that the hyperplane does not contain the symmetry axis. 
    If it is not orthogonal to $e_1$, then there is a set $A_1 \subset \mathbb R$
such that $0<\rho(A_1 \times \mathbb{R}^{n-1} \cap S_1) < \rho( A_1 \times \mathbb{R}^{n-1} \cap S)/2 $
and hence there is a subset $B_1 \subset A_1$ with $\mu(B_1)>0$ for which $f(B_1) \subset (0,1/2)$.
This contradicts optimality, because changing from $f$ to the characteristic function $g$ would lower
the mean-squared error. We conclude that the hyperplane is orthogonal to $e_1$.
        \end{proof}

\subsection{The optimal offset of the separating hyperplane} \label{subsec:offset}

From here on, we consider the dependence of the mean-squared error on the
offset of the separating hyperplane.

We first introduce some additional notation. When the mean-square error is computed,
the measure $\rho$ can be replaced by an effective measure on $\mathbb R$ obtained from projecting
onto the first coordinate.
%\noindent {\bf Objective function statement for two touching $n$-spheres:}\\
%
%Let $C_{\pm}$ denote the centroids of two unit $n$-spheres partitioned by a cutoff on their symmetry axis.
%
%
%\newpage %%PAGE 3

%\noindent {\bf Projected objective function for two touching $n$-spheres:}\\
%{\bf Notation:}
%\noindent Let
We first consider the projection of $\sigma_0$.
With the normalization constant
$$A_n:= \left( \int_{-1}^1 (1-x^2)^{\frac{n-3}{2}} dx \right)^{-1} =\frac{\Gamma (\frac{n}{2} )}{\sqrt{\pi} \Gamma (\frac{n-1}{2} )} ,$$
the resulting measure $\mu_n$ on Borel sets in $[-1,1]$ is given by \cite{MW}
$$d\mu_n (x) :=A_n (1-x^2)^{\frac{n-3}{2}} dx. $$

The probability that $\sigma_0$ 
assigns to $\{x \in \mathbb{S}_0: x_1 \le 1 - a\}$, $a \in [-1,1]$,
is equal to 
the probability of
$\{x \in \mathbb R: x \le 1 - \offset\}$ under $\mu_n$,
$$\Mnminus(\offset ) := \int_{-1}^{1-\offset} d\mu_n (x) .$$
This is the mass of part of the first sphere, obtained by a separating hyperplane between the two centers
of the touching spheres, at a distance of $1-a$ from the center of the first sphere.
From the normalization convention, the total mass of the measure
obtained from two spheres is two, so 
the complementary mass remaining is  
$$\Mnplus(\offset ) := 2-\Mnminus(\offset ) .$$
The mean of the first piece is
$$\Cnminus( \offset ):=\frac{\int_{-1}^{1-\offset} xd\mu_n (x)}{\Mnminus(\offset )} ,$$
and that of the second piece, relative to $C^-_n(0)=0$, accordingly
$$\Cnplus( \offset ):=\frac{2- \int_{-1}^{1-\offset} xd\mu_n (x)}{\Mnplus(\offset )} .$$
With the help of Fubini-Tonelli, the integration over $\mathbb{R}^n$ giving the mean-squared error 
can be reduced
to an integral with respect to $\mu_n$. 
The contributions to the mean-squared error are split into 3 terms,
$$E_- (n, \offset ):=\int_{-1}^{1-\offset} (1-x^2 +(x-\Cnminus( \offset ))^2) d\mu_n (x),$$
$$E_{\pm} (n, \offset ):=\int_{1-\offset}^1  (1-x^2 +(x-\Cnplus( \offset ))^2) d\mu_n (x),$$
 and
$$E_+ (n, \offset ):=\int_{-1}^1  (1-x^2 +(2+x-\Cnplus( \offset ))^2) d\mu_n (x) .$$
 In each of these cases the integrand is the squared distance of a point on either of the two
 spheres from the respective mean of the partition.
 The resulting mean-squared error is obtained by summing the three contributions and dividing by the total mass,
 $$E (n, \offset )=\frac{1}{2} \big[ E_- (n, \offset )+E_{\pm} (n, \offset ) +E_+ (n, \offset )\big] .$$

%\newpage %%PAGE 4
%
%\noindent {\bf Formula motivation:}\\
%$A_n$ and $d\mu_n (x)$ are both given in a paper by [...], and $d\mu_n (x)$ is the probability projection measure of the $n$-sphere onto a one dimensional subspace.\\ \\
%$\Mnminus(\offset )$ is the left mass with the cutoff variable $\offset$ for the case of two touching $n$-spheres. $\Mnplus(\offset )$ is the right mass, and is $2-\Mnminus(\offset )$, as we consider the total mass of the system to be $2$.\\ \\
%$\Cnminus(\offset )$ is the centroid of the left mass, with $\Cnplus(\offset )$ as the right centroid. $C_+$ is constructed with $2-\int_{-1}^{1-\offset} xd\mu_n (x)$ in the numerator as if $\offset =0$, we would have the centroid $C_+$ be at $(2,0)$ (and $C_-$ be at $(0,0)$).\\ \\
%$E_- , E_{\pm} ,$ and $E_+$ are directly from the parts of the $k$-means objective function formulation from page 2. Adding them together and dividing by the total mass of the system produces the objective function, $E$.
%
%
%\newpage %%PAGE 5

\begin{lemma}\label{lem:I} 
Let $n \ge 2$ and $a \in [0,2]$, then
$$E (n, \offset ) = 3-\frac{1}{2} \left( \frac{(\int_{-1}^{1-\offset}  xd\mu_n (x) )^2}{\int_{-1}^{1-\offset}  d\mu_n (x)} + \frac{(2-\int_{-1}^{1-\offset}  xd\mu_n (x) )^2}{2-\int_{-1}^{1-\offset}  d\mu_n (x)} \right) .$$
\end{lemma}
\begin{proof} From normalization, we have the  identities
$\int_{-1}^1  d\mu_n (x) =1 $  and $\int_{-1}^{1-\offset} d\mu_n (x) = 1-\int_{1-\offset}^1 d\mu_n (x)$;
from symmetry,
$\int_{-1}^1 xd\mu_n (x) =0 $  and
$\int_{-1}^{1-\offset}  xd\mu_n (x) = -\int_{1-\offset}^1 xd\mu_n (x)  $.
With the expression for $\Cnminus( \offset )$ and $\Mnminus(\offset )$, 
\begin{align*}
E_- (n, \offset )& = \Mnminus(\offset ) - 2\Cnminus( \offset ) \int_{-1}^{1-\offset}  xd\mu_n (x) + (\Cnminus( \offset ))^2 \Mnminus(\offset )\\
&=\Mnminus(\offset ) - \left( \Cnminus( \offset ) \right)^2 \Mnminus(\offset )
\end{align*}
The integrals in the other terms are expressed similarly, including $\Cnplus( \offset )$ and $\Mnplus(\offset)$ ,
\begin{align*}
E_{\pm} (n, \offset ) &= \int_{1-\offset}^{1} d\mu_n (x) - 2\Cnplus( \offset ) \int_{1-\offset}^{1} xd\mu_n (x) + (\Cnplus( \offset ))^2 \int_{1-\offset}^1 d\mu_n (x) \\
&=1-\Mnminus(\offset ) + 2\Cnplus(\offset )\Cnminus(a) \Mnminus(a) + (\Cnplus(\offset ))^2 (1- \Mnminus(\offset))\\
%&=1-\Mnminus(\offset ) +2\Cnplus(\offset ) \int_{-1}^{1-\offset} xd\mu_n (x) + (\Cnplus(\offset ))^2 \int_{1-\offset}^1 d\mu_n (x) \\
& = 1-\Mnminus(\offset ) +2\Cnplus(\offset )\Cnminus(a) \Mnminus(a)  + (\Cnplus(\offset ))^2 
(\Mnplus(\offset)-1) \, .
\end{align*}

%%PAGE 6
Because the last term is integrated over the entire sphere, the normalization and symmetry yield
\begin{align*}
E_+ (n, \offset ) &= \int_{-1}^{1} d\mu_n (x) - 2(2-\Cnplus( \offset )) \int_{-1}^{1}  xd\mu_n (x) + (2-\Cnplus( \offset ))^2 \int_{-1}^{1}  d\mu_n (x)\\
%&=1-2(2-C_+(n,\offset ))\cdot 0 + (2-C_+(n,\offset ))^2 \cdot 1\\
&=1+ (2-\Cnplus( \offset ))^2\\
%&=1+4-4\Cnplus( \offset ) +(\Cnplus( \offset ))^2\\
&=5-4\Cnplus( \offset ) +(\Cnplus( \offset ))^2 .
\end{align*}

%\newpage %%PAGE 7
%
Adding together $E_- (n,\offset ), E_{\pm} (n, \offset ),$ and $E_+ (n, \offset )$ and dividing by $2$ gives,
after collecting terms,
\begin{align*}
E(n, \offset )=&\frac{1}{2} \big[ \Mnminus(\offset ) - \left( \Cnminus( \offset ) \right)^2 \Mnminus(\offset ) \\
&+1-\Mnminus(\offset ) +2\Cnplus(\offset ) \Cnminus(a) \Mnminus(a) + (\Cnplus(\offset ))^2(\Mnplus(\offset)-1)\\
&+5-4\Cnplus( \offset ) +(\Cnplus( \offset ))^2\big] \\ 
&=\frac{1}{2} \big[ 6-\left( \Cnminus( \offset ) \right)^2 \Mnminus(\offset ) \\
&+2\Cnplus(\offset ) \Cnminus(a) \Mnminus(a) + (\Cnplus(\offset ))^2\Mnplus(\offset)\\
&-4\Cnplus( \offset ) \big]  \, .
\end{align*}
We simplify further by converting between $\Mnminus$ and $\Mnplus$,
\begin{align*}
E(n,\offset)
=&\frac{1}{2} \big[ 6-\left( \Cnminus( \offset ) \right)^2 \Mnminus(\offset ) \\
&+2\Cnplus(\offset ) \left(2-\Cnplus(a)\Mnplus(a)\right) + (\Cnplus(\offset ))^2 \Mnplus(\offset)\\
& -4\Cnplus( \offset )\big] \\
%&=\frac{1}{2} \big[ 6-\left( \Cnminus( \offset ) \right)^2 \Mnminus(\offset ) \\
%&-2\Cnplus(\offset ) \left(2-\int_{-1}^{1-\offset} xd\mu_n (x)\right) + (\Cnplus(\offset ))^2\left(1+\int_{1-\offset}^1 d\mu_n (x)\right)\big] \\ 
%&=\frac{1}{2} \big[ 6-\left( \Cnminus( \offset ) \right)^2 \Mnminus(\offset ) \\
%&-2\Cnplus(\offset ) \left(2-\int_{-1}^{1-\offset} xd\mu_n (x)\right) + (\Cnplus(\offset ))^2\left(2-\Mnminus(\offset) \right)\big] \\ 
=&\frac{1}{2} \big[ 6-\left( \Cnminus( \offset ) \right)^2 \Mnminus(\offset ) \\
&-2(\Cnplus(\offset ))^2 \Mnplus(\offset ) + (\Cnplus(\offset ))^2\Mnplus(\offset )\big] \, .
\end{align*}

%\newpage %%PAGE 8

Thus,
$$E (n,\offset )=3-\frac{1}{2} \left( (C_n^-(\offset ))^2 \Mnminus(\offset ) + (\Cnplus(\offset ))^2 \Mnplus(\offset ) \right).$$
Recalling the values of $C_n^-$, $M_n^-$, $C_n^+$, and $M_n^+$, gives:
$$E (n, \offset ) = 3-\frac{1}{2} \left( \frac{(\int_{-1}^{1-\offset}  xd\mu_n (x) )^2}{\int_{-1}^{1-\offset}  d\mu_n (x)} + \frac{(2-\int_{-1}^{1-\offset}  xd\mu_n (x) )^2}{2-\int_{-1}^{1-\offset}  d\mu_n (x)} \right) .$$
\end{proof}

\begin{lemma} \label{lem:II} The derivative $\frac{\partial}{\partial \offset} E (n,\offset )$ is expressed in terms of $\Mnminus, \Mnplus$ and $a$ as 
\begin{align*}
\frac{\partial}{\partial \offset} E (n,\offset ) =& \frac{2A_n (2\offset -\offset^2 )^{\frac{n-3}{2}}}{(\Mnminus(\offset ) \Mnplus(\offset ))^2} \Big[ (1-\offset ) (\Mnminus(\offset ))^3 \\
&+ (2\offset -1)(\Mnminus( \offset ))^2 \\
&+\frac{A_n}{n-1} (2-\offset )(2\offset -\offset^2 )^{\frac{n-1}{2}} (\Mnminus(\offset ))^2 \\
&+\left( \frac{A_n}{n-1} \right)^2 (2\offset -\offset^2 )^{n-1} \Mnminus(\offset )\\
&+2\frac{A_n}{n-1} (\offset -1)(2\offset -\offset^2 )^{\frac{n-1}{2}} \Mnminus(\offset )\\
&-\left( \frac{A_n}{n-1} \right)^2 (2\offset -\offset^2 )^{n-1} \Big] .
\end{align*}
\end{lemma}

%\newpage %%PAGE 9

\begin{proof} Note that $\int_{-1}^{1-\offset} xd\mu_n (x) = -\frac{A_n}{n-1} (2\offset -\offset^2 )^{\frac{n-1}{2}}$ by direct integration.

Differentiating term by term yields
\begin{align*}
\frac{\partial}{\partial \offset} E (n,\offset ) =&  -\frac{1}{2 (\Mnminus(\offset ))^2} \Big[ 2\frac{A_n^2}{n-1} (1-\offset )(2\offset -\offset^2 )^{n-2} \Mnminus(\offset )\\
&+ \left( \frac{A_n}{n-1} \right)^2 (2\offset -\offset^2 )^{n-1} A_n (2\offset -\offset^2 )^{\frac{n-3}{2}} \Big] \\
&-\frac{1}{2(\Mnplus(\offset ))^2} \Big[ 2\left( 2+\frac{A_n}{n-1} (2\offset -\offset^2 )^{\frac{n-1}{2}} \right) A_n (1-\offset )(2\offset -\offset^2 )^{\frac{n-3}{2}} \Mnplus(\offset )\\
&-  \left( 2+\frac{A_n}{n-1} (2\offset -\offset^2 )^{\frac{n-1}{2}} \right)^2 A_n (2\offset -\offset^2 )^{\frac{n-3}{2}} \Big] \\ 
=& -\frac{A_n (2\offset -\offset^2 )^{\frac{n-3}{2}}}{2 (\Mnminus(\offset ))^2} \Big[ 2\frac{A_n}{n-1} (1-\offset )(2\offset -\offset^2 )^{\frac{n-1}{2}} \Mnminus(\offset )\\
&+ \left( \frac{A_n}{n-1} \right)^2 (2\offset -\offset^2 )^{n-1} \Big] \\
&-\frac{A_n (2\offset -\offset^2 )^{\frac{n-3}{2}}}{2(\Mnplus(\offset ))^2} \Big[ 2\left( 2+\frac{A_n}{n-1} (2\offset -\offset^2 )^{\frac{n-1}{2}} \right) (1-\offset ) \Mnplus(\offset )\\
&-  \left( 2+\frac{A_n}{n-1} (2\offset -\offset^2 )^{\frac{n-1}{2}} \right)^2  \Big] \, .
\end{align*}

Combining terms and simplifying gives
\begin{align*}
\frac{\partial}{\partial \offset} E (n,\offset ) =& \frac{2A_n (2\offset -\offset^2 )^{\frac{n-3}{2}}}{(\Mnminus(\offset )\Mnplus(\offset ))^2} \Big[ -2\frac{A_n}{n-1} (1-\offset )(2\offset -\offset^2 )^{\frac{n-1}{2}} \Mnminus(\offset ) \\
&-\left( \frac{A_n}{n-1} \right)^2 (2\offset -\offset^2 )^{n-1} \\
&+\left( \frac{A_n}{n-1} \right)^2 (2\offset -\offset^2 )^{n-1} \Mnminus(\offset )\\
&+(2\offset -1)(\Mnminus(\offset ))^2\\
&+(1-\offset )(\Mnminus(\offset ))^3 \\
&+\frac{A_n}{n-1} (2-\offset )(2\offset -\offset^2 )^{\frac{n-1}{2}} (\Mnminus(\offset ))^2 \Big] .
\end{align*}

Finally, rearranging terms gives the claimed expression for $\frac{\partial}{\partial \offset} E(n,\offset )$.

%\begin{align*}
%\frac{\partial}{\partial \offset} E (n,\offset ) &= \frac{2A_n (2\offset -\offset^2 )^{\frac{n-3}{2}}}{(\Mnminus(\offset ) \Mnplus(\offset ))^2} \Big[ (1-\offset ) (\Mnminus(\offset ))^3 \\
%&+ (2\offset -1)(\Mnminus( \offset ))^2 \\
%&+\frac{A_n}{n-1} (2-\offset )(2\offset -\offset^2 )^{\frac{n-1}{2}} (\Mnminus(\offset ))^2 \\
%&+\left( \frac{A_n}{n-1} \right)^2 (2\offset -\offset^2 )^{n-1} \Mnminus(\offset )\\
%&+2\frac{A_n}{n-1} (\offset -1)(2\offset -\offset^2 )^{\frac{n-1}{2}} \Mnminus(\offset )\\
%&-\left( \frac{A_n}{n-1} \right)^2 (2\offset -\offset^2 )^{n-1} \Big] .
%\end{align*} 
\end{proof}

%\newpage %%PAGE 12

%By the above work, we have for the objective function for two $n$-dimensional spheres,
%
%$$E (n, \offset ) = 3-\frac{1}{2} \left( \frac{(\int_{-1}^{1-\offset}  xd\mu_n (x) )^2}{\int_{-1}^{1-\offset}  d\mu_n (x)} + \frac{(2-\int_{-1}^{1-\offset}  xd\mu_n (x) )^2}{2-\int_{-1}^{1-\offset}  d\mu_n (x)} \right) \, .$$
%The derivative with respect to $\offset$ is:
%\begin{align*}
%\frac{\partial}{\partial \offset} E (n,\offset ) &= \frac{2A_n (2\offset -\offset^2 )^{\frac{n-3}{2}}}{(\Mnminus(\offset ) \Mnplus(\offset ))^2} \Big[ (1-\offset ) (\Mnminus(\offset ))^3 \\
%&+ (2\offset -1)(\Mnminus( \offset ))^2 \\
%&+\frac{A_n}{n-1} (2-\offset )(2\offset -\offset^2 )^{\frac{n-1}{2}} (\Mnminus(\offset ))^2 \\
%&+\left( \frac{A_n}{n-1} \right)^2 (2\offset -\offset^2 )^{n-1} \Mnminus(\offset )\\
%&+2\frac{A_n}{n-1} (\offset -1)(2\offset -\offset^2 )^{\frac{n-1}{2}} \Mnminus(\offset )\\
%&-\left( \frac{A_n}{n-1} \right)^2 (2\offset -\offset^2 )^{n-1} \Big] .
%\end{align*}
%

To prove that $E (n,\offset)$ is increasing for all $\offset \in (0,2)$ and $n>3$, it suffices to show that $\partial E(n,\offset)/\partial \offset$ is positive for all $\offset \in (0,2)$ and $n>3$.
To this end, we use the simplified expression for $\frac{\partial}{\partial \offset} E (n,\offset )$ and %also prove the following two results stated on subsequent pages.
apply an estimate for $M_n^-$.

%\newpage %%PAGE 13
% III
\begin{lemma} \label{lem:III}
The expression $M_n^-(\offset)$ is continuous in both $n\in [3,\infty )$ and $\offset \in [0,2]$, and $\frac{\partial}{\partial n} \Mnminus(\offset ) >0$ for $n>3$ and $\offset \in (0,1)$ (and negative for $n>3$ and $\offset \in (1,2)$). 
\end{lemma}

\begin{proof}
 First, note that by Leibniz integral rule and integrability of $x^{\alpha} \ln x$, $\alpha >1$, at $0$,
\begin{align*}
\frac{\partial}{\partial n} \int_{-1}^{1-\offset} (1-x^2)^{\frac{n-3}{2}} dx
 &= \int_{-1}^{1-\offset} \frac{\partial}{\partial n} (1-x^2)^{\frac{n-3}{2}} dx
 = \int_{-1}^{1-\offset} \ln (1-x^2) (1-x^2)^{\frac{n-3}{2}} dx \, .
\end{align*}
Thus, taking the partial derivative with respect to $n$, we obtain
$$\frac{\partial}{\partial n} \Mnminus(\offset )=\int_{-1}^{1-\offset} \ln (1-x^2) d\mu_n (x) - \int_{-1}^{1-\offset} d\mu_n (x) \int_{-1}^1 \ln (1-x^2) d\mu_n (x) .$$
Consequently, we have $\frac{\partial}{\partial n} \Mnminus( 0)=\frac{\partial}{\partial n} \Mnminus( 1) = \frac{\partial}{\partial n} \Mnminus( 2) =0$. Next, we show that $\frac{\partial}{\partial n} \Mnminus(\offset ) >0$ for $\offset \in (0,1)$. To see this, we find critical points of $a \mapsto \frac{\partial}{\partial n} \Mnminus(\offset )$.\\

By
$$\frac{\partial}{\partial \offset} \frac{\partial}{\partial n} \Mnminus( \offset ) =\frac{(2\offset -\offset^2 )^{\frac{n-3}{2}}}{\int_{-1}^1 (1-x^2)^{\frac{n-3}{2}} dx} \int_{-1}^1 \left( \ln (1-x^2) - \ln (2\offset -\offset^2 ) \right) d\mu_n (x),$$
we have that $\frac{\partial}{\partial \offset} \frac{\partial}{\partial n} \Mnminus( \offset ) = 0$ if and only if
$$\offset \in \left\{ 0, 1\pm \sqrt{1-\exp \left( \int_{-1}^1 \ln (1-x^2) d\mu_n (x) \right)} ,2 \right\} .$$

%\newpage %%PAGE 14

Hence, for $\offset \in \left( 0,1-\sqrt{1-\exp \left( \int_{-1}^1 \ln (1-x^2) d\mu_n (x) \right)} \right)$, we have $\frac{\partial}{\partial n} \Mnminus( \offset )$ is increasing in $\offset$. Similarly, for $\offset \in \left( 1-\sqrt{1-\exp \left( \int_{-1}^1 \ln (1-x^2) d\mu_n (x) \right)} ,1 \right)$, we have $\frac{\partial}{\partial n} \Mnminus( \offset )$ is decreasing in $\offset$. Therefore, by $\frac{\partial}{\partial n} \Mnminus(0)=\frac{\partial}{\partial n} \Mnminus(1) =0$, we see $\frac{\partial}{\partial n} \Mnminus( \offset ) > 0$ for all $\offset \in (0,1)$. Repeating this for $\offset \in [1,2]$ gives $\frac{\partial}{\partial n} \Mnminus( \offset ) < 0$ for all $\offset \in (1,2)$. Thus we have shown $\Mnminus( \offset )$ is increasing in $n>3$ for $\offset \in (0,1)$ (and decreasing in $n>3$ for $\offset \in (1,2)$).
\end{proof}

\begin{corollary} \label{cor:MnM3}For $\offset \in [0,1]$, we then have the inequalities $M^- _3( \offset ) \le \Mnminus(\offset ) \le 1$ for all $n>3$.
\end{corollary}

%This result allows us to make estimates to show that $\frac{\partial}{\partial \offset} E (n, \offset )> 0$ for $\offset \in (0,1]$ and $n>3$. For $\offset \in [0,1]$, we have the inequalities $M_- (3, \offset ) \le \Mnminus(\offset ) \le 1$ for all $n>3$.\\

% IV
\begin{lemma} \label{lem:IV}
For all $n \ge 3$ and for all $\offset \in [1,2)$, $\Mnminus( \offset ) \ge \frac{A_n}{n-1} \left( 2\offset - \offset^2 \right)^{\frac{n-1}{2}}$.
\end{lemma}

\begin{proof}

We make the change of variables $y=1+x$ with $dy=dx$, in $\Mnminus( \offset )=\int_{-1}^{1-\offset} A_n (1-x^2)^{\frac{n-3}{2}} dx$ to obtain $\Mnminus( \offset ) = \int_0^{2-\offset} A_n (2y-y^2)^{\frac{n-3}{2}} dy$. Repeating integration by parts on parts $(2-y)^{\frac{n-3}{2}}$ and $y^{\frac{n-3}{2}} dy$ yields the formula
\begin{small}
$$\int_0^{2-\offset} A_n \left( 2-y\right)^{\frac{n-3}{2}} y^{\frac{n-3}{2}} dy = 2\frac{A_n}{n-1} \sum_{k=0}^{\infty} \left( \prod_{j=0}^k \frac{n-2j-1}{n+2j-1} \right) (2\offset -\offset^2 )^{\frac{n-(2k+3)}{2}} (2-\offset )^{2k+1} .$$
\end{small}

%PAGE 15

By
\begin{align*}
\Big| \left( \prod_{j=0}^K \frac{n-2j-1}{n+2j-1} \right) \Big| &= \Big| (-1)^K \left( \prod_{j=0}^K \left( 1- \frac{n-1}{j+\frac{1}{2} (n-1)} \right) \right) \Big| \\
&\le \Big| \left( \prod_{j=0}^K \exp \left( -\frac{n-1}{j+\frac{1}{2} (n-1)} \right) \right) \Big| \\
&=\Big| \exp \left( -\sum_{j=0}^K \frac{n-1}{j+\frac{1}{2} (n-1)} \right) \Big|\\
&\rightarrow 0
\end{align*}
 as $K\rightarrow \infty$, we see the alternating series converges. Moreover, since the first term is always positive, the sum converges to a function always greater than zero for $\offset \in (1,2)$ (by property of alternating series). Lastly, we see that for each odd $n\ge 3$, there are exactly $\frac{n-1}{2}$ positive terms and for even $n\ge 4$, there are $\frac{n-2}{2}$ positive terms prior to a convergent alternating series (which starts at a positive term).\\
Consequently,
\begin{align*}
\Mnminus( \offset ) &\ge 2\frac{A_n}{n-1} (2\offset - \offset^2 )^{\frac{n-3}{2}} (2-\offset )\\
&=\frac{2}{\offset} \frac{A_n}{n-1} (2\offset -\offset^2 )^{\frac{n-1}{2}} ,
\end{align*}
which is greater than or equal to $\frac{A_n}{n-1} (2\offset -\offset^2 )^{\frac{n-1}{2}}$ (by maximizing the denominator for $\offset \in [1,2)$). 
\end{proof}

%\newpage %%PAGE 16

Lemma~\ref{lem:IV} gives estimates to show the main inequality of $\frac{\partial}{\partial \offset} E(n, \offset ) >0$ for $\offset \in [1,2)$.\\

Next, we establish  Theorem~\ref{thm:uniqueminimum} using Lemmas~\ref{lem:I} to \ref{lem:IV}.

%\noindent {\bf Theorem I} {\it The inequality $\frac{\partial}{\partial \offset} E (n,\offset )> 0$ holds for all $\offset \in (0,2)$ and $n>3$. Moreover, $E(n,\offset )$ attains a minimum at $\offset =0$, and this minimum is unique.}\\

\begin{proof}[Proof of Theorem~\ref{thm:uniqueminimum}]
We recall the simplified expressions
$$E (n, \offset ) = 3-\frac{1}{2} \left( \frac{(\int_{-1}^{1-\offset}  xd\mu_n (x) )^2}{\int_{-1}^{1-\offset}  d\mu_n (x)} + \frac{(2-\int_{-1}^{1-\offset}  xd\mu_n (x) )^2}{2-\int_{-1}^{1-\offset}  d\mu_n (x)} \right) ,$$
and
\begin{align*}
\frac{\partial}{\partial \offset} E (n,\offset ) =& \frac{2A_n (2\offset -\offset^2 )^{\frac{n-3}{2}}}{(\Mnminus(\offset ) \Mnplus(\offset ))^2} \Big[ (1-\offset ) (\Mnminus(\offset ))^3 \\
&+ (2\offset -1)(\Mnminus( \offset ))^2 \\
&+\frac{A_n}{n-1} (2-\offset )(2\offset -\offset^2 )^{\frac{n-1}{2}} (\Mnminus(\offset ))^2 \\
&+\left( \frac{A_n}{n-1} \right)^2 (2\offset -\offset^2 )^{n-1} \Mnminus(\offset )\\
&+2\frac{A_n}{n-1} (\offset -1)(2\offset -\offset^2 )^{\frac{n-1}{2}} \Mnminus(\offset )\\
&-\left( \frac{A_n}{n-1} \right)^2 (2\offset -\offset^2 )^{n-1} \Big].
\end{align*}

%\newpage %%PAGE 17

To show the desired inequality, we need only show that the factor in brackets is positive:
\begin{align*}
L(n, \offset ) =& \left( (1-\offset )\Mnminus( \offset ) +2\offset -1\right) \left( \Mnminus( \offset ) \right)^2 \\
&+\left( \frac{A_n}{n-1} (2\offset -\offset^2 )^{\frac{n-1}{2}} \Mnminus( \offset ) - 2\frac{A_n}{n-1} (1-\offset ) (2\offset -\offset^2 )^{\frac{n-1}{2}} \right) \Mnminus( \offset )\\
&+\left( \frac{A_n}{n-1} \right)^2 (2\offset - \offset^2 )^{n-1} \Mnminus( \offset ) - \left( \frac{A_n}{n-1} \right)^2 (2\offset -\offset^2 )^{n-1} > 0
\end{align*}
for $\offset \in (0,2)$ and $n>3$.\pagebreak[3]

We distinguish two cases, depending on the value of $\offset$.\\

{\it Case I:} If $\offset \in (0,1)$, by Corollary~\ref{cor:MnM3}, we replace $\Mnminus( \offset )$ with $M_3^- (\offset )=\frac{1}{2} (2-\offset ) $ for all positive terms. That is,
\begin{align*}
L(n,\offset ) \ge 
& \left( (1-\offset )\Mnminus( \offset ) +2\offset -1\right) \left( \Mnminus( \offset ) \right)^2\\
%&\left( (1-\offset )\left( \frac{1}{2} (2-\offset ) \right) +2\offset -1\right) \left( \Mnminus( \offset ) \right)^2 \\
&+ \frac{A_n}{n-1} (2\offset -\offset^2 )^{\frac{n-1}{2}} \left( \frac{1}{2} (2-\offset ) \right)^2 \\
&- 2\frac{A_n}{n-1} (1-\offset ) (2\offset -\offset^2 )^{\frac{n-1}{2}}  \Mnminus( \offset )\\
&+\left( \frac{A_n}{n-1} \right)^2 (2\offset - \offset^2 )^{n-1} \left( \frac{1}{2} (2-\offset ) \right) - \left( \frac{A_n}{n-1} \right)^2 (2\offset -\offset^2 )^{n-1} .
\end{align*}

Moreover, we see $  (1-\offset )\Mnminus( \offset ) +2\offset -1 \ge 
(1-a)M_3^-(a)+2 \offset -1  = a/2 +a^2/2 \ge 0$.
% by 
%$$f' (n, \offset ) = -\Mnminus(a)-(1-a)A_n(2a-a^2)^{(n-3)/2} +  \ge 0$$ 
%and $f(n, 0) =0$ (note the fact that $f(n, \offset ) \ge 0$ holds for all $\offset \in [0,2]$ and $n\ge 3$).

%% (1/2)(1-a)(2-a) + 2a -1 = 1 - a/2 -a + a^2/2 +2 a -1 = a/2 + a^2/2
%\newpage %%PAGE 18

Hence, the first term can be estimated as well by eliminating $M_n^-(a)$, resulting in the lower bound
\begin{align*}
L(n,\offset ) \ge 
&\left( \frac a 2 + \frac{a^2}{2}\right) \left( \frac{1}{2} (2-\offset )\right)^2 \\
&+\frac{A_n}{n-1} (2\offset -\offset^2 )^{\frac{n-1}{2}} \left( \frac{1}{2} (2-\offset ) \right)^2 \\
&- 2\frac{A_n}{n-1} (1-\offset ) (2\offset -\offset^2 )^{\frac{n-1}{2}}  \Mnminus( \offset )\\
&+\left( \frac{A_n}{n-1} \right)^2 (2\offset - \offset^2 )^{n-1} \left( \frac{1}{2} (2-\offset ) \right) \\
&- \left( \frac{A_n}{n-1} \right)^2 (2\offset -\offset^2 )^{n-1} .
\end{align*}

By Lemma~\ref{lem:III}, we also have that $\Mnminus( \offset ) \le \Mnminus(0)=1$.\\
Using this estimate for the remaining negative factor multiplying $\Mnminus$ gives a further lower bound
from which all quantities other than $a$ have been eliminated,
\begin{align*}
L(n,\offset ) \ge 
&\frac{1}{8} (1+\offset )\offset (2-\offset )^2 \\
&+\frac{A_n}{n-1} (2\offset -\offset^2 )^{\frac{n-1}{2}} \left( \frac{1}{2} (2-\offset ) \right)^2 \\
&- 2\frac{A_n}{n-1} (1-\offset ) (2\offset -\offset^2 )^{\frac{n-1}{2}} \\
&+\left( \frac{A_n}{n-1} \right)^2 (2\offset - \offset^2 )^{n-1} \left( \frac{1}{2} (2-\offset ) \right) \\ 
&- \left( \frac{A_n}{n-1} \right)^2 (2\offset -\offset^2 )^{n-1} \\ 
=&\frac{1}{8} (1+\offset )\offset (2-\offset )^2 \\
&+\frac{1}{8} \frac{A_n}{n-1} (2\offset^4 - 8\offset^3 +24\offset^2 -16\offset ) (2\offset - \offset^2 )^{\frac{n-1}{2}} \\
&-\frac{1}{2} \left( \frac{A_n}{n-1} \right)^2 \offset (2\offset -\offset^2 )^{n-1} .
\end{align*}

%\newpage %%PAGE 19

\noindent Finally, by the second and third term decreasing in $\offset \in (0,1)$, we have
\begin{align*}
L(n,\offset ) &\ge \frac{1}{8} (1+\offset )\offset (2-\offset )^2 + \frac{1}{4} \frac{A_n}{n-1} -\frac{1}{2} \left( \frac{A_n}{n-1} \right)^2 \\
&= \frac{1}{8} \left( (1+\offset )\offset (2-\offset )^2 +\frac{2A_n}{n-1} - \left( \frac{2A_n}{n-1} \right)^2 \right) \\
&\ge \frac{1}{8}  (1+\offset )\offset (2-\offset )^2 
> 0 .
\end{align*}

Consequently for $\offset \in (0,1]$, $\frac{\partial}{\partial \offset} E (n, \offset ) > 0$.

%\newpage %%PAGE 20
\medskip
\noindent {\it Case II:} If $\offset \in [1,2)$, we re-examine $L(n,\offset )$ and apply Lemma~\ref{lem:IV}.\\

By $\frac{A_n}{n-1} (2\offset -\offset^2 )^{\frac{n-1}{2}} \ge \left( \frac{A_n}{n-1} \right)^2 (2\offset -\offset^2 )^{n-1}$, we have
\begin{align*}
L(n,\offset ) =& \left( (1-\offset )\Mnminus( \offset ) +2\offset -1\right) \left( \Mnminus(\offset ) \right)^2 \\
&+\left( \frac{A_n}{n-1} (2\offset -\offset^2 )^{\frac{n-1}{2}} \Mnminus(\offset ) - 2\frac{A_n}{n-1} (1-\offset ) (2\offset -\offset^2 )^{\frac{n-1}{2}} \right) \Mnminus( \offset )\\
&+\left( \frac{A_n}{n-1} \right)^2 (2\offset - \offset^2 )^{n-1} \Mnminus(\offset ) - \left( \frac{A_n}{n-1} \right)^2 (2\offset -\offset^2 )^{n-1} \\ 
\ge& \left( (1-\offset )\Mnminus( \offset ) +2\offset -1\right) \left( \Mnminus(\offset ) \right)^2 \\
&+\left( \frac{A_n}{n-1} \right)^2 (2-\offset )(2\offset -\offset^2 )^{n-1} \left( \Mnminus(\offset )\right)^2 +\left( \frac{A_n}{n-1} \right)^2 (2\offset -\offset^2 )^{n-1} \Mnminus( \offset )\\
&+2\left( \frac{A_n}{n-1} \right)^2 (\offset -1)(2\offset -\offset^2 )^{n-1} \Mnminus( \offset ) -\left( \frac{A_n}{n-1} \right)^2 (2\offset -\offset^2 )^{n-1} \\ 
=&\left( (1-\offset )\Mnminus( \offset ) +2\offset -1\right) \left( \Mnminus(\offset ) \right)^2\\
&+\left( \frac{A_n}{n-1} \right)^2  (2\offset )\left[ \Mnminus( \offset )-\frac{1}{2} \left( \Mnminus( \offset )\right)^2  \right] (2\offset -\offset^2 )^{n-1} \\
&+\left( \frac{A_n}{n-1} \right)^2 \left[ 2\left( \Mnminus( \offset )\right)^2 - \Mnminus( \offset ) -1 \right] (2\offset -\offset^2 )^{n-1} .
\end{align*}

Using Lemma~\ref{lem:IV} in the last inequality and recalling that 
if $a \in (1,2)$, then $\Mnminus(a) < \Mnminus(1) = 1/2$, we have
$$ (1-\offset )\Mnminus( \offset ) +2\offset -1 >
     \frac 3  2 \offset - \frac 1 2  > 0 
     $$ which gives
\begin{align*}
L(n,\offset ) \ge&  \left( (1-\offset )\Mnminus( \offset ) +2\offset -1\right)  \left( \frac{A_n}{n-1} \right)^2 (2\offset -\offset^2 )^{n-1} \\
&+\left( \frac{A_n}{n-1} \right)^2  (2\offset )\left[ \Mnminus( \offset )-\frac{1}{2} \left( \Mnminus( \offset )\right)^2  \right] (2\offset -\offset^2 )^{n-1} \\
&+\left( \frac{A_n}{n-1} \right)^2 \left[ 2\left( \Mnminus( \offset )\right)^2 - \Mnminus( \offset ) -1 \right] (2\offset -\offset^2 )^{n-1} .
\end{align*}

%\newpage %%PAGE 21

Thus, combining terms, we obtain a lower bound%that $L(n,\offset )$ is greater than or equal to
\begin{align*}
L(n,\offset)  \ge &\left( \frac{A_n}{n-1} \right)^2 \Big[ (1-\offset )\Mnminus(\offset ) +2\offset -1\\
&+2\offset \Mnminus(\offset ) - \offset \left( \Mnminus(\offset ) \right)^2 \\
& +2\left( \Mnminus(\offset ) \right)^2 -\Mnminus(\offset ) -1 \Big] (2\offset -\offset^2 )^{n-1} \\ 
=&\left( \frac{A_n}{n-1} \right)^2 \left[ \offset \Mnminus(\offset ) +2(\offset -1) +(2-\offset )\left( \Mnminus(\offset ) \right)^2 \right] (2\offset -\offset^2 )^{n-1} ,
\end{align*}
consisting of strictly positive terms if $1<a<2$.

Consequently, we see for $n>3$ and $\offset \in (1,2)$,
\begin{align*}
\frac{\partial}{\partial \offset} E (n, \offset ) >0
\end{align*}

We conclude that for $\offset \in (0,2)$ and $n>3$, $E(n,\offset)$ is strictly increasing, thus attaining its unique minimum at $\offset = 0$. 
\end{proof}


\begin{thebibliography}{99}


\bibitem{BUWI} { J.A. Bucklew and  G.L. Wise},  Multidimensional asymptotic quantization theory with $r^{th}$ power distortion measures, {\em IEEE Trans. Inform. Theory} {\bf  28}(2):239--247, 1982. 
\bibitem{BER}{T. Berger}, {\em Rate Distortion Theory}, Prentice-Hall, Englewood Cliffs, New Jersey, 1971.
\bibitem{Das99}
S. Dasgupta,
\newblock Learning mixtures of gaussians,
\newblock in: {\em 40th Annual Symposium
  on Foundations of Computer Science 1999}, pages 634--644. IEEE, 1999.
\bibitem{DuFG99}
Q.~Du, V.~Faber, and M.~Gunzburger,
\newblock Centroidal {V}oronoi tessellations: Applications and algorithms,
\newblock {\em SIAM Review}, 41(4):637--676, 1999.

\bibitem{GEGR} {A. Gersho and R.M. Gray}, {\em Vector Quantization and Signal Compression}, 
Springer International Series in Engineering and Computer Science, Springer, {\bf 159}, Berlin, pp.~732, 1991.

\bibitem{GRLU} {S. Graf and H.  Luschgy}, {\em Foundations of Quantization for Probability Distributions}. Lecture
Notes in Math. 1730. Springer, Berlin, pp. 203, 2000.



\bibitem{iguchi2015tightness}
T.~Iguchi, D.~G. Mixon, J.~Peterson, and S.~Villar.
\newblock On the tightness of an {SDP} relaxation of k-means,
\newblock {\em arXiv preprint arXiv:1505.04778}, 2015.

\bibitem{IguchiMPV15}
T.~Iguchi, D.~G. Mixon, J.~Peterson, and S.~Villar.
\newblock Probably certifiably correct k-means clustering,
\newblock {\em Mathematical Programming}, pages 1--38, 2015.

  \bibitem{KIE} {J. C. Kieffer}, Exponential rate of convergence for Lloyd's method~I, {\em IEEE Trans. on Inform. Theory}, Special issue on quantization,  {\bf 28}(2):205--210, 1982.
 \bibitem{LLOYD} {S. P. Lloyd}, Least squares quantization in PCM. {\em IEEE Trans. Inform. Theory} {\bf 28}(2):129--137, 1982.
 
 \bibitem{LLLSW}
 X. Li, Y. Li, S. Ling, T. Strohmer, and K. Wei, When Do Birds of a Feather Flock Together? K-Means, Proximity, and Conic Programming,
 {\em arXiv preprint arXiv:1710.06008}, 2017.
 
 \bibitem{LuZ16}
Y.~Lu and H.~H. Zhou.
\newblock Statistical and computational guarantees of {L}loyd's algorithm and
  its variants,
\newblock {\em arXiv preprint arXiv:1612.02099}, 2016.

\bibitem{MacKay}
D. MacKay,  {\em Information Theory, Inference and Learning Algorithms}, Cambridge University Press, Cambridge, 2003.

\bibitem{MW}
C. E. Mueller and F. B. Weissler, Hypercontractivity for the heat semigroup for ultraspherical polynomials and on the n-sphere, {\em Journal of Functional Analysis} {\bf 48} (2): 252--283, 1982.

\bibitem{PengW07}
J.~Peng and Y.~Wei.
\newblock Approximating k-means-type clustering via semidefinite programming,
\newblock {\em SIAM Journal on Optimization}, 18(1):186--205, 2007.


\bibitem{POL} {D. Pollard},  A central limit theorem for $k$-means clustering, {\em Ann. Probab.}, {\bf 10}, 919--926, 1982.

\bibitem{Roychowdhury}
M. K. Roychowdhury, Optimal quantizers for some absolutely continuous probability measures,
{\em arXiv preprint arXiv:1608.03815}, 2016.



\bibitem{Selim84}
S.~Z. Selim and M.~A. Ismail.
\newblock k-means-type algorithms: A generalized convergence theorem and
  characterization of local optimality,
\newblock {\em IEEE Transactions on pattern analysis and machine intelligence},
  6(1):81--87, 1984.

\bibitem{Steinhaus}
H. Steinhaus, Sur la division des corps mat{\'e}riels en parties, {\em Bull. Acad. Polon. Sci.} {\bf  4} (12): 801--804, 1957.

\bibitem{Vattani11}
A.~Vattani.
\newblock k-means requires exponentially many iterations even in the plane,
\newblock {\em Discrete and Computational Geometry}, 45(4):596--616, 2011.



\end{thebibliography}
\end{document}